\documentclass[10pt]{article}
\usepackage[affil-it]{authblk}
\usepackage{amsmath,amsthm,amssymb,cite}
\usepackage{lineno,hyperref}

\usepackage{calrsfs}

\usepackage{tikz}
\usetikzlibrary{arrows,chains,matrix,positioning,scopes}
\makeatletter
\tikzset{join/.code=\tikzset{after node path={%
\ifx\tikzchainprevious\pgfutil@empty\else(\tikzchainprevious)%
edge[every join]#1(\tikzchaincurrent)\fi}}}
\makeatother
\tikzset{>=stealth',every on chain/.append style={join},
         every join/.style={->}}

\newtheorem{theorem}{Theorem}[section]
\newtheorem{lemma}[theorem]{Lemma}
\newtheorem{corollary}[theorem]{Corollary}
\theoremstyle{definition}

\title{Strong Homology Theory of Continuous Maps {\footnote{The author was supported by grant FR/233/5-103/14 from Shota Rustaveli National Science Foundation (SRNSF)}}}
\date{\vspace{-5ex}}

\author{Anzor Beridze and Vladimer Baladze}
\affil{Department of Mathematics\\Batumi Shota Rustaveli  State University}
\providecommand{\keywords}[1]{\textbf{\textit{Keywords and Phrases:}} #1}
\begin{document}
\maketitle

\begin{abstract}The current work is motivated by the papers \cite{16}, \cite{18}, \cite{1901}, \cite{1902}. In particular, using Theorem 3.7 of \cite{16} and methods developed in this paper, the spectral and strong homology groups of continuous maps were defined and studied \cite{18}, \cite{1901}, \cite{1902}. In this paper is shown that strong homology groups of continuous maps are a homology type functor, which is a strong shape invariant and has the semi-continuous property. Besides, the new axioms and the conjecture of the uniqueness of the constructed functor is formulated. 
\end{abstract}

\section*{Introduction}

As is known, the singular homology theory is an exact homology theory but having no continuity property, while the $\check{C}$ech homology
theory is a continuous homology but not exact in general. The strong (Steenrod) homology theory is an exact homology theory that possesses the modified continuity, the so-called semi-continuity property \cite{80}, \cite{9001}.

The $\check{C}$ech and Steenrod homology theories are related through
a Milnor exact sequence \cite{80}, \cite{9001}. The analogous result is obtained for strong homology groups of continuous maps in the paper \cite{1902}. The aim of the present paper is to show that the strong homology groups of continuous maps are a homology type functor, which is a strong shape invariant, has the semi-continuous property and is related through a Milnor exact sequence to the $\check{C}$ech homology functor of continuous maps. Besides, we will formulate the new axiomatic system for the constructed theory and the conjecture on the uniqueness theorem.

\section{Some Facts and Notations}

Throughout the paper the following notation is used:

\begin{itemize}
\item $\mathbf{Top_{CM}}$ - the category of compact metric spaces and continuous
maps;
\item $\mathbf{MO}{{\mathbf{R}}_{\mathbf{Top_{CM}}}}$ - the category of morphisms of the category $\mathbf{Top_{CM}}$;
\item $ \mathbf p :f\to \mathbf f$\textbf{ }- a strong expansion of a continuous map $f:X \to Y$ \cite{301};
\item $(\varphi ,{\varphi }'):\mathbf{f}\to \mathbf{f}'$ - a coherent mapping of inverse sequences \cite{301};
\item $[(\varphi ,{\varphi }')]:\mathbf{f}\to \mathbf{f}'$ - the coherent homotopy class of a coherent mapping $(\varphi ,{\varphi }'):\mathbf{f}\to \mathbf{f}'$ \cite{301};
\item The coherent homotopy category $\mathbf{CH}(\mathbf{tow}-\mathbf{Mo}{{\mathbf{r}}_{\mathbf{CM}}})$ - category of all inverse sequences of continuous maps of compact metric spaces and coherent homotopy classes $[(\mathbf{\varphi },\mathbf{{\varphi }'})]$ of a coherent morphisms $(\mathbf{\varphi },\mathbf{{\varphi }'})$ \cite{301};
\item  $\mathbf{CH}(\mathbf{tow}-\mathbf{Mo}{{\mathbf{r}}_{\mathbf{ANR}}})$ - the full subcategory of the category $\mathbf{CH}(\mathbf{tow}-\mathbf{Mo}{{\mathbf{r}}_{\mathbf{CM}}})$, the objects of which are inverse sequences of $\text{ANR}$-maps \cite{301};
\item The strong fiber shape category $\mathbf{SSh}(\mathbf{Mo}{{\mathbf{r}}_{\mathbf{CM}}})$ - the category of all continuous maps of compact metric spaces and all strong shape morphisms \cite{301};
\item $H_n(f)$ - the spectral homology gruop of continuous map $f:X \to Y$ \cite{18};
\item $\mathbf{S}:\mathbf{MO}{{\mathbf{R}}_{\mathbf{Top_{CM}}}}\to \mathbf{SSh}(\mathbf{MO}{{\mathbf{R}}_{\mathbf{Top_{CM}}}})$ - the strong fiber shape functor \cite{301};
\item $Ch$ - the category of chain complexes and chain maps;
\item $Mor_{Ch}$ - the category of chain maps and morphisms of chain
maps;
\item The chain cone $C_{*}(f_{\#})$  of a chain map $f_{\#}:L_{*}\to M_{*}$, whose definition differs somewhat from the standard definition, i.e., $C_{*}(f_{\#})=\lbrace C_{n}(f_{\#}),{\partial }\rbrace $ is the chain complex, where $C_{n}(f_{\#})\approx L_{n-1} \oplus M_{n}, \forall n \in \mathbb{N}$ and ${\partial}(l,m)=(\partial (l),-\partial (m)+f_{\#}(l)), $ $\forall (l,m)\in C_{n}(f_{\#})$;
\item A coherent morphism \textbf{ }$\Phi : f_{\#}\to g_{\#}$ of chain
maps: $\Phi =\lbrace (\phi ^{1},\phi ^{2}),\phi ^{1,2}\rbrace $ is
a system, where $\phi ^{1}:L_{*}\to P_{*}$ and $\phi ^{2}:M_{*}\to Q_{*}$ are chain maps and $\phi ^{1,2}:L_{*}\to Q_{*}$ is a chain homotopy of the chain maps $g_{\#}\phi ^{1}$ and $\phi ^{2}f_{\#}$ \cite{1901};

\item The chain map $\Phi _{\#}:C_{*}( f_{\#})\to C_{*}(g_{\#})$ induced by a coherent morphism $\Phi : f_{\#}\to g_{\#}$: $\Phi _{\#}$ is defined by the formula
\end{itemize}
\begin{center} $\Phi _{\#}(l,m)=(\phi ^{1}(l),\phi ^{2}(m)+\phi ^{1,2}(l) )$
\end{center}

\cite{1901};

\begin{itemize}
\item A coherent homotopy $D=\lbrace (D^{1},D^{2}),D^{1,2}\rbrace $ of
coherent morphisms $\Phi =\lbrace (\phi ^{1},\phi ^{2}),\phi
^{1,2}\rbrace $ and $\Psi =\lbrace (\psi ^{1},\psi ^{2}),\psi
^{1,2}\rbrace $: $D=\lbrace (D^{1},D^{2}),D^{1,2}\rbrace $ is a system,
where $D^{1}$ is a chain homotopy of $\phi ^{1}$ and $\psi ^{1}$, $
D^{2}$ is also a chain homotopy of $\phi ^{2}$ and $\psi ^{2}$ and $
D^{1,2}:L_{*}\to Q_{*}$ is a chain map of degree two, which satisfies the
following conditions:
\end{itemize}
\begin{center}$\partial D^{1,2}-D^{1,2}\partial =g_{\#}D^{1}-D^{2}f_{\#}+\phi ^{1,2}-\psi
^{1,2}$\\
\end{center}

\cite{1901};

\begin{itemize}
\item The chain homotopy $D_{\#}:C_{*}( f_{\#})\to C_{*}(g_{\#})$ induced by a coherent homotopy $D$ of coherent morphisms $\Phi $ and $\Psi $
: $D_{\#}$ is defined by the formula
\end{itemize}
\begin{center}$D_{\#}(l,m)=(D^{1}(l),-D^{2}(m)+D^{1,2}(l) )$
\end{center} 

\cite{1901}.

\section{Strong Homology Functor}

\ \

First of all, we will define the strong homology functor from the category $\mathbf{CH}(\mathbf{tow}-\mathbf{Mo}{{\mathbf{r}}_{\mathbf{ANR}}})$ to the category $\mathbf{Ab}$  of direct sequences of Abelian groups. Let $\mathbf{f}=\{{{f}_{i}},({{p}_{i,i+1}},{p'}_{i,i+1}),\mathbb{N}\}:\mathbf{X \to X'}$ be any object of the category $\mathbf{CH}(\mathbf{tow}-\mathbf{Mo}{{\mathbf{r}}_{\mathbf{ANR}}})$. Let $f_{i}^{\#}:{{S}_{*}}({{X}_{i}})\to {{S}_{*}}(X_{i}^{'})$ be the singular chain map induced by ${{f}_{i}}:{{X}_{i}}\to X_{i}^{'}$. Consider the pair $(p_{i,i+1}^{\#},{p'}_{i,i+1}^{\#})$, where $p_{i,i+1}^{\#}:{{S}_{*}}({{X}_{i+1}})\to {{S}_{*}}({{X}_{i}})$ and ${p'}_{i,i+1}^{\#}:{{S}_{*}}(X_{i+1}^{'})\to {{S}_{*}}(X_{i}^{'})$ are the chain maps induced by ${{p}_{i,i+1}}:{{X}_{i+1}}\to {{X}_{i}}$ and ${p'}_{i,i+1}:X_{i+1}^{'}\to X_{i}^{'}$, respectively. For the morphism  $({{p}_{i,i+1}},{p'}_{i,i+1}):{{f}_{i+1}}\to {{f}_{i}}$ we have
$${{f}_{i}}\cdot {{p}_{i,i+1}}={p'}_{i,i+1}\cdot {{f}_{i+1}}.$$                                        
So $f_{i}^{\#}\cdot p_{i,i+1}^{\#}={p'}_{i,i+1}^{\#}\cdot f_{i+1}^{\#}$. Therefore $(p_{i,i+1}^{\#},{p'}_{i,i+1}^{\#}):f_{i+1}^{\#}\to f_{i}^{\#}$ is a morphism of chain maps. So we obtain the following inverse sequences of chain complexes
$${{S}_{*}}(\mathbf{X})=\{ {{S}_{*}}({{X}_{i}}),p_{i,i+1}^{\#}, \mathbb{N} \} ,$$                                  
 $${{S}_{*}}(\mathbf{{X}'})=\{ {{S}_{*}}(X_{i}^{'}),{p'}_{i,i+1}^{\#}, \mathbb{N} \} ,$$                                
 $${{C}_{*}}(\mathbf{f})=\{ {{C}_{*}}(f_{i}^{\#}),(p_{i,i+1}^{\#},{p'}_{i,i+1}^{\#}), \mathbb{N} \},$$                        
where ${{C}_{*}}(f_{i}^{\#})$ is the chain cone of the chain map $f_{i}^{\#}:{{S}_{*}}({{X}_{i}})\to {{S}_{*}}(X_{i}^{'})$.  Let for each $n\in \mathbb{N}$
$${{K}_{n}}(\mathbf{X})=\mathop{\prod }^{}{{S}_{n}}({{X}_{i}}),$$                                         
$${{K}_{n}}(\mathbf{{X}'})=\mathop{\prod }^{}{{S}_{n}}(X_{i}^{'}),$$                                          
$${{K}_{n}}(\mathbf{f})=\mathop{\prod }^{}{{C}_{n}}({{f}_{i}}),$$                                           
and ${{\partial }_{n}}:{{K}_{n}}(\mathbf{X})\to {{K}_{n-1}}(\mathbf{X}),$ ${{\partial' }_{n}}:{{K}_{n}}(\mathbf{{X}'})\to {{K}_{n-1}}(\mathbf{{X}'})$  and  ${\tilde {{\partial }}_{n}}:{{K}_{n}}(\mathbf{f})\to {{K}_{n-1}}(\mathbf{f})$  are induced by ${{\partial }_{n}}:{S_{n}}(X_i)\to {S_{n-1}}(X_i),$ ${{\partial' }_{n}}:{S_{n}}(X'_i)\to {S_{n-1}}(X'_i)$  and  ${\tilde{{\partial }}_{n}}:{C_{n}}(f_i)\to {C_{n-1}}(f_i)$. It is clear that ${{K}_{*}}( \mathbf{X} )=\{{{K}_{n}}( \mathbf{X} ),\partial  \}$,  ${{K}_{*}}( {\mathbf{{X}'}} )=\{{{K}_{n}}( {\mathbf{{X}'}} ),\partial'  \}$ and ${{K}_{*}}( \mathbf{f} )=\{{{K}_{n}}( \mathbf{f} ),\tilde{\partial } \}$ are chain complexes.  Consider the maps ${{\mathbf{p}}_{\#}}:{{K}_{*}}( \mathbf{X} )\to {{K}_{*}}( \mathbf{X} ),$  $\mathbf{p'}_{\#}^:{{K}_{*}}( {\mathbf{{X}'}} )\to {{K}_{*}}( {\mathbf{{X}'}} )$  and $( {{\mathbf{p}}_{\#}},\mathbf{p}_{\#}^{'} ):{{K}_{*}}( \mathbf{f} )\to {{K}_{*}}( \mathbf{f} ),$  defined by
$${{\mathbf{p}}_{\#}}( {{c}^{n}} )={{\mathbf{p}}_{\#}} \{c_{i}^{n} \}=\{ p_{i,i+1}^{\#}( c_{i+1}^{n} )-c_{i}^{n} \},~~~\forall ~{{c}^{n}}\in {{K}_{n}}( \mathbf{X} ).$$             
$$\mathbf{p'}_{\#}( {{c'}^{n}} )=\mathbf{p'}_{\#}\{ {c'}_{i}^{n} \}=\{ p_{i,i+1}^{'\#}( {c'}_{i+1}^{n} )-{c'}_{i}^{n} \},~~~\forall ~{{c'}^{n}}\in {{K}_{n}}( {\mathbf{{X}'}} ).$$             
	$$( {{\mathbf{p}}_{\#}},\mathbf{p'}_{\#} )( {{{\tilde{c}}}^{n}} )=( {{\mathbf{p}}_{\#}},\mathbf{p'}_{\#} )( \{ {{{\tilde{c}}}^{n}_i} \} )=( {{\mathbf{p}}_{\#}},\mathbf{p}_{\#}^{'} )( \{ c_{i}^{n-1},{c'}_{i}^{n} \} )=$$
$$( \{p_{i,i+1}^{\#}( c_{i+1}^{n-1} )-c_{i}^{n-1} , p_{i,i+1}^{'\#}( {c'}_{i+1}^{n} )-{c'}_{i}^{n} \} )  ~~~~\forall ~ {{{\tilde{c}}}^{n}} \in {{K}_{n}}( \mathbf{f} ).$$

\begin{lemma}
For each inverse sequence $\mathbf{f}= \{ {{f}_{i}},({{p}_{i,i+1}},{p'}_{i,i+1})~,\mathbb{N} \}$ of continuous maps of topological spaces ${{\mathbf{p}}_{\#}}:{{K}_{*}}( \mathbf{X} )\to {{K}_{*}}( \mathbf{X} ),$  $\mathbf{p'}_{\#}:{{K}_{*}}( {\mathbf{{X}'}} )\to {{K}_{*}}( {\mathbf{{X}'}} )$  and $( {{\mathbf{p}}_{\#}},\mathbf{p'}_{\#} ):{{K}_{*}}( \mathbf{f} )\to {{K}_{*}}( \mathbf{f} )$ are chain maps.
\end{lemma}
\begin{proof}
Let $c^n\in {K_n}\left({\mathbf X}\right)$, then 
\[{\partial }_n{{\mathbf p}}_{{\mathbf \#}}\left(c^n\right)={\partial }_n{{\mathbf p}}_{{\mathbf \#}}\left\{c^n_i\right\}={\partial }_n\left\{p^{\#}_{i,i+1}\left(c^n_{i+1}\right)-c^n_i\right\}=\] 
\[\left\{{\partial }_np^{\#}_{i,i+1}\left(c^n_{i+1}\right)-{\partial }_n(c^n_i)\right\}=\left\{{p^{\#}_{i,i+1}\partial }_n\left(c^n_{i+1}\right)-{\partial }_n(c^n_i)\right\}=\] 
\[{{\mathbf p}}_{{\mathbf \#}}\left\{{\partial }_n(c^n_i)\right\}={{\mathbf p}}_{{\mathbf \#}}{\partial }_n\left(c^n\right).\] 
So ${{\mathbf{p}}_{\#}}:{{K}_{*}}( \mathbf{X} )\to {{K}_{*}}( \mathbf{X} )$ is a chain map. The same way we can show that $\mathbf{p'}_{\#}:{{K}_{*}}( {\mathbf{{X}'}} )\to {{K}_{*}}( {\mathbf{{X}'}} )$ is chain map as well. Therefore, the third part of the lemma remains to be proved. Let $\tilde{c}^n\in {{K}}_{n}\left({\mathbf f}\right)$, then 
\[ \tilde{{\partial}}_n \left( {\mathbf p}_{\#},{\mathbf p'}_{\#} \right) \left(\tilde{c}^n\right)= \tilde{{\partial}}_n( {{\mathbf{p}}_{\#}},\mathbf{p'}_{\#} )( \{ c_{i}^{n-1},{c'}_{i}^{n} \} )=\]
\[\tilde{{\partial}}_n( \{p_{i,i+1}^{\#}( c_{i+1}^{n-1} )-c_{i}^{n-1} , {p'}_{i,i+1}^{\#}( {c'}_{i+1}^{n} )-{c'}_{i}^{n} \} )=\]
\[( \tilde{{\partial}}_n\{p_{i,i+1}^{\#}( c_{i+1}^{n-1} )-c_{i}^{n-1}, {p'}_{i,i+1}^{\#}( {c'}_{i+1}^{n} )-{c'}_{i}^{n} \} )= \]
\[ ( \{ {{\partial}}_n (p_{i,i+1}^{\#}( c_{i+1}^{n-1} )-c_{i}^{n-1}),-{{\partial}}_n ({p'}_{i,i+1}^{\#}( {c'}_{i+1}^{n} )-{c'}_{i}^{n}) )+f^{\#}_i(p_{i,i+1}^{\#}( {c}_{i+1}^{n-1} )-{c}_{i}^{n-1})\} ).\]
On the other hand,
\[ \left( {\mathbf p}_{\#},{\mathbf p'}_{\#} \right) \left( \tilde{{\partial}}_n (\tilde{c}^n) \right)= ( {{\mathbf{p}}_{\#}},\mathbf{p'}_{\#} )( \tilde{{\partial}}_n \{ c_{i}^{n-1},{c'}_{i}^{n} \} )=\]
\[ ( {{\mathbf{p}}_{\#}},\mathbf{p'}_{\#} )(  \{ {\partial}_n ({c}_{i}^{n-1}),-{\partial}_n({c'}_{i}^{n})+f^{\#}_i ({c}_{i}^{n-1}) \} )=\]
\[ ( \{ p_{i,i+1}^{\#}( {\partial}_{n-1}(c_{i+1}^{n-1})) -{\partial}_{n-1}(c_{i}^{n-1}), 
{p'}_{i,i+1}^{\#}(-{\partial}_{n}( {c'}_{i+1}^{n})+f^{\#}_{i+1}(c_{i+1}^{n-1}))+{\partial}_{n}( {c'}_{i}^{n})-f^{\#}_{i}(c_{i}^{n-1})\} )=\]
\[ ( \{ {\partial}_{n-1}(p_{i,i+1}^{\#}(c_{i+1}^{n-1})) -{\partial}_{n-1}(c_{i}^{n-1}), -{\partial}_{n}({p'}_{i,i+1}^{\#}( {c'}_{i+1}^{n}))+{p'}_{i,i+1}^{\#}(f^{\#}_{i+1}(c_{i+1}^{n-1}))+{\partial}_{n}( {c'}_{i}^{n})-f^{\#}_{i}(c_{i}^{n-})\} )=\]
\[ ( \{ {\partial}_{n-1}(p_{i,i+1}^{\#}(c_{i+1}^{n-1})) -{\partial}_{n-1}({c}_{i}^{n-1}), -{\partial}_{n}({p'}_{i,i+1}^{\#}( {c'}_{i+1}^{n}))+{\partial}_{n}( {c'}_{i}^{n})+f^{\#}_{i}(p_{i,i+1}^{\#}(c_{i+1}^{n-1}))-f^{\#}_{i}(c_{i}^{n-1})\} )=\]
\[ ( \{ {\partial}_{n-1}(p_{i,i+1}^{\#}(c_{i+1}^{n-1}) -c_{i}^{n-1}), -{\partial}_{n}({p'}_{i,i+1}^{\#}( {c'}_{i+1}^{n})- {c'}_{i}^{n})+f^{\#}_{i}(p_{i,i+1}^{\#}(c_{i+1}^{n-1})-c_{i}^{n-1})\} ).\]

\end{proof}

Let ${{C}_{*}}( {{\mathbf{p}}_{\#}} )$, ${{C}_{*}}( \mathbf{p'}_{\#})$ and ${{C}_{*}}( {{\mathbf{p}}_{\#}},\mathbf{p'}_{\#} )$ be the chain cones of chain maps ${{\mathbf{p}}_{\#}}$, $\mathbf{p'}_{\#}$ and $( {{\mathbf{p}}_{\#}},\mathbf{p'}_{\#} )$, respectively. Consider the maps
$$\sigma :{{C}_{*}}( \mathbf{p'}_{\#} )\to {{C}_{*}}( {{\mathbf{p}}_{\#}},\mathbf{p'}_{\#}),$$                                    
$$\partial :{{C}_{*}}( {{\mathbf{p}}_{\#}},\mathbf{p'}_{\#} )\to {{C}_{*}}( {{\mathbf{p}}_{\#}} ) ,$$                                     
where for each $n\in N$,  ${{\sigma }}:{{C}_{n}}( \mathbf{p'}_{\#})\to {{C}_{n}}( {{\mathbf{p}}_{\#}},\mathbf{p'}_{\#})$ and ${{\partial }}:{{C}_{n}}( {{\mathbf{p}}_{\#}},\mathbf{p'}_{\#})\to {{C}_{n-1}}( {{\mathbf{p}}_{\#}} )$ are defined by
	$${{\sigma }}( {{c'}^{n-1}},{{c'}^{n}} )={{\sigma }_{n}}( \{{c'}_{i}^{n-1}\}~,\{ {c'}_{i}^{n} \} )=$$
$$=( \{ 0, {c'}_{i}^{n-1} \},\{ 0,{c'}_{i}^{n}\})~~~\forall ( {{c'}^{n-1}},{{c'}^{n}} )\in {{C}_{n}}( \mathbf{p'}_{\#}),~~~$$

	$${{\partial }}( {{{\tilde{c}}}^{n-1}},{{{\tilde{c}}}^{n}} )={{\partial }_{n}}( \{c_{i}^{n-2},{c'}_{i}^{n-1}\}~,\{c_{i}^{n-1},{c'}_{i}^{n}\} )=$$
		$$=( \{c_{i}^{n-2} \}~, \{ c_{i}^{n-1} \} )~~~\forall ( {{{\tilde{c}}}^{n-1}},{{{\tilde{c}}}^{n}} )\in {{C}_{n}}( {{\mathbf{p}}_{\#}},\mathbf{p'}_{\#}).$$

\begin{lemma} For each inverse sequence $\mathbf{f}= \{ {{f}_{i}},({{p}_{i,i+1}},{p'}_{i,i+1})~, \mathbb{N} \}$  of continuous maps of topological spaces the following short sequence
	$$0\to {{C}_{*}}( \mathbf{p'}_{\#} )\overset{\sigma }{\mathop{\to }}\,{{C}_{*}}( {{\mathbf{p}}_{\#}},\mathbf{p'}_{\#} )\overset{\partial }{\mathop{\to }}\,{{C}_{*}}( {{\mathbf{p}}_{\#}} )\to 0$$
is exact.
\end{lemma}

\begin{proof}
Let $( {{{c'}_1}^{n-1}},{{{c'}_1}^{n}} ), ( {{{c'}_2}^{n-1}},{{{c'}_2}^{n}} )\in {{C}_{n}}( \mathbf{p'}_{\#})$ be such elements that
$${{\sigma }}( {{c'_1}^{n-1}},{{c'_1}^{n}} )=
{{\sigma }}( {{c'_2}^{n-1}},{{c'_2}^{n}} ).$$
So we have
$$( \{ 0, {c'}_{1,i}^{n-1} \},\{ 0,{c'}_{1,i}^{n}\})=( \{ 0, {c'}_{2,i}^{n-1} \},\{ 0,{c'}_{2,i}^{n}\}).$$
The last equation means that for each $n,i \in \mathbb{N}$ we have ${c'}_{1,i}^{n-1}={c'}_{2,i}^{n-1}$ and ${c'}_{1,i}^{n}={c'}_{2,i}^{n}$. Therefore 
$( {{c_1}^{n-1}},{{c_1}^{n}} )=( {{c_2}^{n-1}},{{c_2}^{n}}).$ So $\sigma$ is a monomorphism.

Let $( {{{c}}^{n-2}},{{{c}}^{n-1}} )=( \{{c}_{i}^{n-2}\}~,\{{c}_{i}^{n-1}\} ) \in {{C}_{n-1}}( \mathbf{p}_{\#})$ be any element. consider $( {{{\tilde{c}}}^{n-1}},{{{\tilde{c}}}^{n}} )=( \{ c_{i}^{n-2},0 \}~,\{c_{i}^{n-1},0 \} ) \in {{C}_{n}}( {{\mathbf{p}}_{\#}},\mathbf{p'}_{\#}).$ In this case we have 
   	$${{\partial }}( {{{\tilde{c}}}^{n-1}},{{{\tilde{c}}}^{n}} )={{\partial }_{n}}( \{c_{i}^{n-2},0\}~,\{c_{i}^{n-1},0 \} )=( \{c_{i}^{n-2} \}~, \{ c_{i}^{n-1} \} )=( {{{c}}^{n-2}},{{{c}}^{n-1}} ).$$
So $\partial$ is an epimorphism.

Now lets show that $\partial \cdot \sigma = 0$. Indeed
$$\partial\left( {{\sigma }} \left( {{c'}^{n-1}},{{c'}^{n}} \right) \right)=\partial \left({{\sigma }} \left( \{{c'}_{i}^{n-1}\}~,\{ {c'}_{i}^{n} \} \right) \right)=$$
$$\partial \left( \{ 0, {c'}_{i}^{n-1} \},\{ 0,{c'}_{i}^{n}\} \right)$$   
$$= \left( \{ 0 \}~, \{ 0 \} \right)=\left(  0 ~,  0  \right)=0.$$

To the end of the proof it remains to show that $Ker(\partial) \subset Im(\sigma)$. Consider any element $( {{{\tilde{c}}}^{n-1}},{{{\tilde{c}}}^{n}} )\in Ker(\partial)$. So we have 
$${{\partial }}( {{{\tilde{c}}}^{n-1}},{{{\tilde{c}}}^{n}} )={{\partial }_{n}}( \{c_{i}^{n-2},{c'}_{i}^{n-1}\}~,\{c_{i}^{n-1},{c'}_{i}^{n}\} )=$$
		$$=( \{c_{i}^{n-2} \}~, \{ c_{i}^{n-1} \} )=0.$$
So for each $n,i \in N$ we have $c_{i}^{n-2}=0$ and $c_{i}^{n-1}=0$. Consider the element $({c'}^{n-1},{c'}^n)=( \{{c'}_{i}^{n-1}\}~,\{{c'}_{i}^{n}\} )$. In this case we have 

$${{\sigma }}( {{c'}^{n-1}},{{c'}^{n}} )={{\sigma }}( \{{c'}_{i}^{n-1}\}~,\{ {c'}_{i}^{n} \} )=$$
$$=( \{ 0, {c'}_{i}^{n-1} \},\{ 0,{c'}_{i}^{n}\})=( {{{\tilde{c}}}^{n-1}},{{{\tilde{c}}}^{n}} ).$$
\end{proof}
 By Lemma 2.2 we obtain the following long exact sequence  
$$\cdots \to {{\text{H}}_{\text{n}+1}}( {{C}_{*}}( \mathbf{p'}_{\#} ) )\overset{{{\sigma }_{*}}}{\mathop{\to }}\,{{\text{H}}_{\text{n}+1}}( {{C}_{*}}( {{\mathbf{p}}_{\#}},\mathbf{p'}_{\#} ) )\overset{{{\partial }_{*}}}{\mathop{\to }}\,{{\text{H}}_{\text{n}}}( {{C}_{*}}( {{\mathbf{p}}_{\#}} ) )\overset{E}{\mathop{\to }}\,{{\text{H}}_{\text{n}}}( {{C}_{*}}( {{\mathbf{q}}_{\#}} ) )\to \cdots .~~~~$$ 

By Theorem 3.1 of \cite{1901}, if we consider the strong  ANR-expansion $\mathbf{X}=\{X_i,p_{i,i+1}, \mathbb{N} \}$ of compact metric space $X\in \mathbf{CM}$, then $( n+1 )$-dimensional homology group   ${{H}_{n+1}}( {{C}_{*}}( {{\mathbf{p}}_{\#}} ) )$ of the chain cone ${{C}_{*}}( {{\mathbf{p}}_{\#}} )$ of the chain map ${{\mathbf{p}}_{\#}}:{{K}_{*}}( \mathbf{X} )\to {{K}_{*}}( \mathbf{X} )$  is isomorphic to the strong homology group ${{\mathbf{\bar{H}}}_{\mathbf{n}}}( X )$ (Steenrod homology) of  $X$. On the other hand, by Corollary 2.2 of \cite{1902}, if $\mathbf{f}=\{ f_i,( \varphi_{i,i+1},{\varphi '}_{i,i+1} ), \mathbb{N} \}$ is strong ANR-expansion of $ f:X \to X'$, then $( n+1 )$-dimensional homology group   ${{H}_{n+1}}( {{C}_{*}}( {{\mathbf{p}}_{\#}},{{\mathbf{p'}}_{\#}} ) )$ of the chain cone ${{C}_{*}}( {{\mathbf{p}}_{\#}},{{\mathbf{p'}}_{\#}} )$ of the chain map $({{\mathbf{p}}_{\#}},{{\mathbf{p'}}_{\#}}):{{K}_{*}}( \mathbf{f} )\to {{K}_{*}}( \mathbf{f} )$  is isomorphic to the strong homology group ${{\mathbf{\bar{H}}}_{\mathbf{n}}}( f )$ (Steenrod homology) of  $f$. That is why, we will denote   the $( n+1 )$-dimensional homology group of complexes ${{C}_{*}}( {{\mathbf{p}}_{\#}} ),$ ${{C}_{*}}( \mathbf{p'}_{\#})$ and $ {{C}_{*}}( {{\mathbf{p}}_{\#}},\mathbf{p'}_{\#})$ by the symbols  ${{\mathbf{{H}}}_{\mathbf{n}}}( \mathbf {X} ),$ ${{\mathbf{{H}}}_{\mathbf{n}}}( {{ \mathbf{X}}'} )$ and ${{\mathbf{{H}}}_{\mathbf{n}}}( \mathbf{f} )$, respectively. Lets denote  the following long exact  sequence
$$
\cdots \to {{\mathbf{{H}}}_{\mathbf{n}}}( { \mathbf{X'}} )\overset{{{\sigma }_{*}}}{\mathop{\to }}\,{{\mathbf{{H}}}_{\mathbf{n}}}( \mathbf{f} )\overset{{{\partial }_{*}}}{\mathop{\to }}\,{{\mathbf{{H}}}_{\mathbf{n}-1}}( \mathbf{X} )\overset{E}{\mathop{\to }}\,{{\mathbf{{H}}}_{\mathbf{n}-1}}( { \mathbf{X'}} )\to \cdots
$$
by  ${\mathbf{{H}}}( \mathbf{f} )$, which is  an object of the category $\mathbf{Ab}$. Our aim is to show that $\mathbf{{H}}$ is a functor from the category $\mathbf{CH}( \mathbf{tow}-\mathbf{Mo}{{\mathbf{r}}_{\mathbf{ANR}}} )$ to the category $\mathbf{Ab}$.

Consider any morphism $\left[ ( \mathbf{\varphi },\mathbf{{\varphi }'} ) \right]:\mathbf{f}\to \mathbf{g}$ of the category $\mathbf{CH}( \mathbf{tow}-\mathbf{Mo}{{\mathbf{r}}_{\mathbf{ANR}}} ).$ Let $( \mathbf{\varphi },\mathbf{{\varphi }'} )$ be any representative of the coherent homotopy class $\left[ ( \mathbf{\varphi },\mathbf{{\varphi }'} ) \right]$. Therefore, $( \mathbf{\varphi },\mathbf{{\varphi }'} )$ is a system $\{ ( {{\varphi }_{m}},{\varphi'} _{m}),( {{\varphi }_{m,m+1}},{\varphi '} _{m,m+1} ),\varphi  \}~:\mathbf{f}\to \mathbf{g}$, where $\varphi ~: \mathbb{N} \to \mathbb{N}$ is an increasing function, $( {{\varphi }_{m}},{\varphi'} _{m} ):{{f}_{\varphi ( m )}}\to {{g}_{m}}$ and $( {{\varphi }_{m,m+1}},{\varphi'} _{m,m+1} ):{{f}_{\varphi ( m+1 )}}\times {{1}_{I}}\to {{g}_{m}}$ are  morphisms such that 
$$( {{\varphi }_{m,m+1}}( x,0 ),{{\varphi}'}_{m,m+1}( {x}',0 ) )=( {{\varphi }_{m}}( {{p}_{\varphi ( m ),\varphi ( m+1 )}}( x ) ),{{\varphi}'} _{m}( {p'}_{\varphi ( m ),\varphi ( m+1 )}( {{x}'} ) ) ),$$ 
$$( {{\varphi }_{m,m+1}}( x,1 ),{{\varphi}'} _{m,m+1}( {x}',1 ) )=( {{q}_{m,m+1}}( {{\varphi }_{m+1}}( x ) ),{q'}_{m,m+1}( {{\varphi}'} _{m+1}( {{x}'} ) ) ).$$            
Consider the corresponding coherent mappings $\mathbf {\varphi }=\{ {{\varphi }_{m}},{{\varphi }_{m,m+1}},\varphi  \}~:\mathbf{X}\to \mathbf{Y}$ and $\mathbf{{\varphi }'}= \{ \varphi _{m}^{'},\varphi _{m,m+1}^{'},\varphi  \}~:\mathbf{{X}'}\to \mathbf{{Y}' }$ of inverse sequences of topological spaces. By the Lemma 2.2 of  \cite{1901}  given mappings induce  a coherent chain maps ${{\mathbf{\varphi }}_{\#}}= \{ ( \mathbf{\varphi }_{\#}^{0},\mathbf{\varphi }_{\#}^{0} ),\mathbf{\varphi }_{\#}^{1} \}~:{{\mathbf{p}}_{\#}}\to {{\mathbf{q}}_{\#}}$ and ${{\mathbf{\varphi' }}_{\#}}= \{ ( \mathbf{\varphi' }_{\#}^{0},\mathbf{\varphi' }_{\#}^{0} ),\mathbf{\varphi' }_{\#}^{1} \}~:{{\mathbf{p'}}_{\#}}\to {{\mathbf{q'}}_{\#}}$, which themselves  induce the chain maps  ${{\mathbf{\varphi }}_{\#}}:{{C}_{*}}( {{\mathbf{p}}_{\#}} )\to {{C}_{*}}( {{\mathbf{q}}_{\#}} )$ and ${{\mathbf{\varphi' }}_{\#}}:{{C}_{*}}( {{\mathbf{p'}}_{\#}} )\to {{C}_{*}}( {{\mathbf{q'}}_{\#}} )$. So, for each coherent morphism $( \mathbf{\varphi },\mathbf{{\varphi' }} )= \{ ( {{\varphi }_{m}},{{\varphi'}} _{m} ),( {{\varphi }_{m,m+1}},{{\varphi}'} _{m,m+1} ),\varphi  \}~:\mathbf{f}\to \mathbf{g}$ we have the chain maps
$${{\mathbf{\varphi }}_{\#}}:{{C}_{*}}( {{\mathbf{p}}_{\#}} )\to {{C}_{*}}( {{\mathbf{q}}_{\#}} ) ,$$                              
$$\mathbf{\varphi' }_{\#}:{{C}_{*}}( \mathbf{p'}_{\#} )\to {{C}_{*}}( \mathbf{q'}_{\#} ) .$$                              
Let $( {{\mathbf{\varphi }}_{\#}},\mathbf{{\varphi}'}_{\#} ):{{C}_{*}}( {{\mathbf{p}}_{\#}},\mathbf{p'}_{\#})\to {{C}_{*}}( {{\mathbf{q}}_{\#}},\mathbf{q'}_{\#})$ be the chain map defined by
	$$( {{\mathbf{\varphi }}_{\#}},\mathbf{{\varphi}'}_{\#})( {{{\tilde{c}}}^{n-1}},{{{\tilde{c}}}^{n}} )=( {{\mathbf{\varphi }}_{\#}},\mathbf{{\varphi}'}_{\#} )( { \tilde{c}_{i}^{n-1} },{ \tilde{c}_{i}^{n} } )=$$
	$$( {{\mathbf{\varphi }}_{\#}},\mathbf{{\varphi}'}_{\#} )( \{c_{i}^{n-2},{c'}_{i}^{n-1}\}~,\{c_{i}^{n-1},{c'}_{i}^{n}\} )=$$
$$( {{\mathbf{\varphi }}_{\#}}(\{c_{i}^{n-2}, c_{i}^{n-1} \}),\mathbf{{\varphi}'}_{\#}(\{{c'}_{i}^{n-1},{c'}_{i}^{n} \}~) ) .$$                     
In this case the following diagram is commutative
$$
\begin{matrix}
   0\to {{C}_{*}}\left( \mathbf{p'}_{\#} \right)\overset{\sigma }{\mathop{\to }}\,{{C}_{*}}\left( {{\mathbf{p}}_{\#}},\mathbf{p'}_{\#} \right)\overset{\partial }{\mathop{\to }}\,{{C}_{*}}\left( {{\mathbf{p}}_{\#}} \right)\to 0  \\
   ~~~~\downarrow \mathbf{{\varphi}' }_{\#~}~~~~~~\downarrow \left( {{\mathbf{\varphi }}_{\#}},\mathbf{{\varphi}' }_{\#} \right)~~~\downarrow {{\mathbf{\varphi }}_{\#}}  \\
   0\to {{C}_{*}}\left( \mathbf{q'}_{\#} \right)\overset{\sigma }{\mathop{\to }}\,{{C}_{*}}\left( {{\mathbf{q}}_{\#}},\mathbf{q'}_{\#} \right)\overset{\partial }{\mathop{\to }}\,{{C}_{*}}\left( {{\mathbf{q}}_{\#}} \right)\to 0.  \\
\end{matrix}
$$
On the other hand, the obtained diagram induces morphism between the long exact sequences,which is denoted by ${{( \mathbf{\varphi },\mathbf{{\varphi }'} )}_{*}}:\mathbf{{H}}( \mathbf{f} ) \to  \mathbf{{H}}( \mathbf{g} )$. On the other hand, by Corollary 2.4 \cite{1901} and the Lemma of five homomorphisms any two representative $( {{\mathbf{\varphi }}_{1}},\mathbf{\varphi' }_{1} )$ and $( {{\mathbf{\varphi }}_{2}},\mathbf{\varphi' }_{2})$ of a coherent homotopy class $\left[ ( \mathbf{\varphi },\mathbf{{\varphi }'} ) \right]:\mathbf{f}\to \mathbf{g}$ of the category $\mathbf{CH}( \mathbf{tow}-\mathbf{Mo}{{\mathbf{r}}_{\mathbf{ANR}}} )$ induce the same morphism
$${{( {{\mathbf{\varphi }}_{1}},\mathbf{{\varphi}' }_{1} )}_{*}}={{( {{\mathbf{\varphi }}_{2}},\mathbf{{\varphi}' }_{2} )}_{*}}:\mathbf{{H}}( \mathbf{f} )\to \mathbf{{H}}( \mathbf{g} ).$$                 
Therefore, we can say that any coherent homotopy class $\left[ ( \mathbf{\varphi },\mathbf{{\varphi }'} ) \right]$ induces  the morphism  ${{\left[ ( \mathbf{\varphi },\mathbf{{\varphi }'} ) \right]}_{*}}:\mathbf{\bar{H}}( \mathbf{f} )\to \mathbf{\bar{H}}{{( \mathbf{g} )}}$, which can be defined by
$${{\left[ ( \mathbf{\varphi },\mathbf{{\varphi }'} ) \right]}_{*}}={{( {{\mathbf{\varphi }}},\mathbf{{\varphi}' })}_{*}},~~~\forall ~~( {{\mathbf{\varphi }}},\mathbf{{\varphi}' } )\in \left[ ( \mathbf{\varphi },\mathbf{{\varphi }'} ) \right].$$            
So, we define the homological functor
$$\mathbf{{H}}:\mathbf{CH}( \mathbf{tow}-\mathbf{Mo}{{\mathbf{r}}_{\mathbf{ANR}}} )\to \mathbf{Ab}.$$                     
Using the constructed functor, we can define the so called strong homology functor $\mathbf{{H}}:\mathbf{SSH}( \mathbf{Mor}_{\mathbf{CM}} )\to \mathbf{Ab}$ in the following way: For each morphism $f\in \mathbf{Mo}{{\mathbf{r}}_{\mathbf{CM}}}$ consider the corresponding strong expansion $( \mathbf{p},\mathbf{{p}'} ):f\to \mathbf{f}$. Let denote the homology sequence $\mathbf{{H}}( \mathbf{f} )$ by $\mathbf{\bar{H}}( f )$ and call it strong homology sequence of $f$.  It is clear that $\mathbf{\bar{H}}( f )$ does not defend on the choice of the expansion $( \mathbf{p},\mathbf{{p}'} )$. In the same way, for each strong shape morphism $F:f\to g$, consider the corresponding triple $( ( \mathbf{p},\mathbf{p'} ),( \mathbf{q},\mathbf{q'} ),\left[ ( \mathbf{\psi },\mathbf{{\psi }'} ) \right] )$, where $( \mathbf{p},\mathbf{{p}'} ):f\to \mathbf{f}$ and $( \mathbf{q},\mathbf{{q}'} ):g\to \mathbf{g}$ are strong fiber expansions and $\left[ ( \mathbf{\psi },\mathbf{{\psi }'} ) \right]\mathbf{f}\to \mathbf{g}$ is a coherent homotopy class of the coherent morphisms. Let denote corresponding induced morphism  ${{\left[ ( \mathbf{\varphi },\mathbf{{\varphi }'} ) \right]}_{*}}:\mathbf{{H}}( \mathbf{f} )\to \mathbf{{H}}{{( \mathbf{g} )}_{*}}$ by ${{F}_{*}}$. It is called induced morphisms by strong shape morphism $F:f\to g$. Therefore, we obtain the so called strong homological functor:
$$\mathbf{{H}}:\mathbf{SSH}( \mathbf{Mo}{{\mathbf{r}}_{\mathbf{CM}}} )\to \mathbf{Ab}. $$                      
By using the obtained functor, we can define the homology functor
$$\mathbf{\bar{H}}:{\mathbf {Mor}}_{\mathbf{CM}}\to \mathbf{Ab}$$                             
as a composition $\mathbf{\bar{H}}=\mathbf{{H}}\cdot \mathbf{S}$, where 
$$\mathbf{S}:~{\mathbf{Mor}}_{\mathbf{CM}}\to \mathbf{SSH}( \mathbf{Mo}{{\mathbf{r}}_{\mathbf{CM}}} )$$                       
is the strong shape functor. Note that for the homology functor  $\mathbf{\bar{H}}: \mathbf{Mor}_{CM}\to \mathbf{Ab}$ we have the following:

\begin{corollary} 
For each continuous map $f\in {\mathbf{Mor}}_{\mathbf{CM}}$ of compact metric spaces the corresponding homological sequence
	$$\cdots \to {{\mathbf{\bar{H}}}_{\mathbf{n}}}( {{X}'} )\overset{{{\sigma }_{*}}}{\mathop{\to }}\,{{\mathbf{\bar{H}}}_{\mathbf{n}}}( f )\overset{{{\partial }_{*}}}{\mathop{\to }}\,{{\mathbf{\bar{H}}}_{\mathbf{n}-1}}( X )\overset{E}{\mathop{\to }}\,{{\mathbf{\bar{H}}}_{\mathbf{n}-1}}( {{X}'} )\to \cdots $$
is exact.

\end{corollary}
\begin{corollary} 
If any two morphisms  $( {{\varphi }_{1}},{\varphi'} _{1} ),( {{\varphi }_{2}},{\varphi'} _{2} ):f\to g$ induce the same strong shape morphisms, then
$${{( {{\varphi }_{1}},{\varphi'} _{1} )}_{*}}={{( {{\varphi }_{2}},{\varphi'} _{2} )}_{*}}:\mathbf{H}( f )\to \mathbf{H}( g ).$$ \qed               
\end{corollary}
\begin{theorem} 
If $( {\varphi }_,{\varphi'} ):f\to g$  is such a morphism that $f$ and $g$ are inverse limits of inverse ANR-sequences  $\mathbf{f}= \{ {{f}_{i}},({{p}_{i,i+1}},p_{i,i+1}^{'})~,\mathbb{N} \}$ and $\mathbf{g}= \{ {{g}_{i}},({{q}_{i,i+1}},q_{i,i+1}^{'})~,\mathbb{N} \}$ and the pair $( {\varphi }_,{\varphi'} ):f\to g$ is the pair of inverse limits on inverse ANR-sequence ${\mathbb {\varphi}}= \{ {{\varphi}_{i}},({{p}_{i,i+1}},q_{i,i+1}^{'})~,\mathbb{N} \} : { \mathbf{X} \to \mathbf{Y}}$ and ${\mathbf{\varphi'}}= \{ {{\varphi}_{i}},({{p}_{i,i+1}},q_{i,i+1}^{'})~,\mathbb{N} \}: { \mathbf{X'} \to \mathbf{Y'}}$, then the following sequence 
$$
\begin{matrix}
   0\to Li{{m}^{1}}{{H}_{n+1}}( {{f}_{i}} )\to {{\mathbf{H}}_{\mathbf{n}}}( f )\to Lim~{{H}_{n}}( {{f}_{i}} )~\to 0  \\
   ~~~~~~~\downarrow Li{{m}^{1}} \left( {{\mathbf{\varphi_{i+1} }}},\mathbf{{\varphi'}_{i+1} } \right)_{*}~\downarrow \left( {{\mathbf{\varphi }}},\mathbf{{\varphi}' } \right)_{*}~\downarrow Lim \left( {{\mathbf{\varphi }}}_{i},{\mathbf{{\varphi}'}_{i} } \right)_{*} \\
   0\to Li{{m}^{1}}{{H}_{n+1}}( {{g}_{i}} )\to {{\mathbf{H}}_{\mathbf{n}}}( g )\to Lim~{{H}_{n}}( {{g}_{i}} )~\to 0  \\
\end{matrix}
$$
is exact, where ${{H}_{n}}( {{f}_{i}} )$ and ${{H}_{n}}( {{g}_{i}} )$ are spectral homology groups of ${{f}_{i}}: {X_i \to {X'}_i}$ and $g_i: {Y_I \to {Y'}_i} $, respectively.
\end{theorem}
\begin{proof} Let $\left( {{\mathbf{\varphi }}},\mathbf{{\varphi}' } \right) : { \mathbf{f}} \to {\mathbf{g}}$ be the morphism given in the theorem. In this case there exists a corresponding commutative diagram:
$$
\begin{matrix}
   0\to K_{*} \left( \mathbf{f} \right)\overset{\sigma }{\mathop{\to }}\,{{C}_{*}}\left( {{\mathbf{p}}_{\#}},\mathbf{p'}_{\#} \right)\overset{\partial }{\mathop{\to }}\,K_{*}\left( \mathbf{f} \right)\to 0  \\
   ~~~~\downarrow \prod (\mathbf{\varphi_{i}})_{\#}~~~~~~\downarrow \left( {{\mathbf{\varphi }}_{\#}},\mathbf{{\varphi}' }_{\#} \right)~~~\downarrow \prod (\mathbf{\varphi_{i}})_{\#}  \\
   0\to K_{*} \left( \mathbf{g} \right)\overset{\sigma }{\mathop{\to }}\,{{C}_{*}}\left( {{\mathbf{q}}_{\#}},\mathbf{q'}_{\#} \right)\overset{\partial }{\mathop{\to }}\,K_{*}\left( \mathbf{g} \right)\to 0.  \\
\end{matrix}
$$
That induces the following commutative diagram:
$$
\begin{matrix}
   \cdots \to H_{n+1}(\mathbf K_{*}( \mathbf{f})){{\stackrel{(\mathbf p_{*},\mathbf {p'}_{*})}{\longrightarrow}}} H_{n+1}(\mathbf K_{*}( \mathbf{f})){{\stackrel{\sigma
_{*}}{\longrightarrow}}}H_{n+1}(C_{*}(\mathbf p_\#,{\mathbf {p'}}_{\#}))  {{\stackrel{\pi _{*}}{\longrightarrow}}} \\
   ~~~~~~~~~~~~~~~~~~\downarrow \prod (\mathbf{\varphi_{i}})_{*}~~~~~~~~~~~~~~\downarrow \prod (\mathbf{\varphi_{i}})_{*}~~~~~~~~~\downarrow \left( {{\mathbf{\varphi }}_{*}},\mathbf{{\varphi}' }_{*} \right)~~~~~~~~~~~~~~~~ \\
   \cdots \to H_{n+1}(\mathbf K_{*}( \mathbf{g})){{\stackrel{(\mathbf q_{*},\mathbf {q'}_{*})}{\longrightarrow}}} H_{n+1}(\mathbf K_{*}( \mathbf{g})){{\stackrel{\sigma
_{*}}{\longrightarrow}}}H_{n+1}(C_{*}(\mathbf q_\#,{\mathbf {q'}}_{\#}))  {{\stackrel{\pi _{*}}{\longrightarrow}}} \\
\end{matrix}
$$
~~~~~~~~~~~~~~~~~~~~~~~~~~~~~~~~~~~~~~~~~~~~~~~~~~~~~~~~$
\begin{matrix}
   {{\stackrel{\pi _{*}}{\longrightarrow}}}{H_{n}(\mathbf K_{*}( \mathbf{f}))}{{\stackrel{(\mathbf p_{*},\mathbf {p'}_{*})}{\longrightarrow}}}H_{n}(\mathbf K_{*}( \mathbf{f}))\to \cdots   \\
   ~~~~~~~~\downarrow \prod (\mathbf{\varphi_{i}})_{*} ~~~~~~~~~\downarrow \prod (\mathbf{\varphi_{i}})_{*} \\
   {{\stackrel{\pi _{*}}{\longrightarrow}}}{H_{n}(\mathbf K_{*}( \mathbf{g}))}{{\stackrel{(\mathbf q_{*},\mathbf {q'}_{*})}{\longrightarrow}}}H_{n}(\mathbf K_{*}( \mathbf{g}))\to \cdots~~.  \\
\end{matrix}
$

On the other hand, $H_{n}(\mathbf K_{*}( \mathbf{f}))=
H_{n}(\prod{S_{*}(f_{i})})\approx \prod{H_{n}(S_{*}(f_{i})) } =\prod{H_{n}(f_{i}})$ and $H_{n}(\mathbf K_{*}( \mathbf{g}))=
H_{n}(\prod{S_{*}(g_{i})})\approx \prod{H_{n}(S_{*}(g_{i})) } =\prod{H_{n}(g_{i}})$.
Hence, we obtain 
$$
\begin{matrix}
   \cdots \to \prod H_{n+1}(f_i){{\stackrel{(\mathbf p_{*},{\mathbf p'}_{*})}{\longrightarrow}}} H_{n+1}(f_i){{\stackrel{\sigma
_{*}}{\longrightarrow}}} \mathbf{H_{n}}(f){{\stackrel{\pi _{*}}{\longrightarrow}}}  \\
   ~~~~~~~~~~~~~~~~~~~~\downarrow \prod (\mathbf{\varphi_{i}})_{*}~~~~~~~~~\downarrow \prod (\mathbf{\varphi_{i}})_{*}~~~~~~~~~\downarrow \left( {{\mathbf{\varphi }}},\mathbf{{\varphi}' } \right)_{*}~~~~~~~~~~~~~~ \\
   \cdots \to \prod H_{n+1}(g_i){{\stackrel{(\mathbf q_{*},{\mathbf q'}_{*})}{\longrightarrow}}} H_{n+1}(g_i){{\stackrel{\sigma
_{*}}{\longrightarrow}}} \mathbf{H_{n}}(g){{\stackrel{\pi _{*}}{\longrightarrow}}}\\
\end{matrix}
$$
~~~~~~~~~~~~~~~~~~~~~~~~~~~~~~~~~~~~~~~~~~~~~~~~~~~~~~~$
\begin{matrix}
   {{\stackrel{\pi _{*}}{\longrightarrow}}}{H_{n}(f_i}){{\stackrel{(\mathbf p_{*},\mathbf {p'}_{*})}{\longrightarrow}}}H_{n}(f_i)\to \cdots  \\
   ~~~~~~~\downarrow \prod (\mathbf{\varphi_{i}})_{*} ~~~~~~\downarrow \prod (\mathbf{\varphi_{i}})_{*} \\
  {{\stackrel{\pi _{*}}{\longrightarrow}}}{H_{n}(g_i}){{\stackrel{(\mathbf q_{*},\mathbf {q'}_{*})}{\longrightarrow}}}H_{n}(g_i)\to \cdots . \\
\end{matrix}
$

By the definition of the chain map $ (\mathbf{p_{\#}},\mathbf{{p'}_{\#}}): \mathbf {K_{*}( \mathbf {f})} \to  \mathbf {K_{*}( \mathbf {f})}$ we have 

$$\left( p_{*},{p'}_{*} \right) \lbrace \left[ c_i^n,{c'}_i^{n+1} \right] \rbrace=\left( p_{*},{p'}_{*} \right) \left[ \lbrace c_i^n,{c'}_i^{n+1} \rbrace \right]= \left[ \left( p_{*},{p'}_{*} \right) \left( \lbrace c_i^n \rbrace, \lbrace {c'}_i^{n+1} \rbrace \right) \right]=$$
$$\left[ p_{\#} \lbrace c_i^n \rbrace ,{p'}_{\#} \lbrace {c'}_i^{n+1} \rbrace) \right] =\left[ \lbrace p_{i,i+1}^{\#}(c_{i+1}^n)-{c}_i^n \rbrace, \lbrace {p'}_{i,i+1}^{\#} \left( c_{i+1}^{n+1} \right) - {c'}_i^{n+1} \rbrace \right]=$$
$$\left( \lbrace p_{i,i+1}^{*} \left[ c_{i+1}^n \right]- \left[ {c}_i^n \right]  \rbrace, \lbrace {p'}_{i,i+1}^{*}\left[ c_{i+1}^{n+1} \right] - \left[ {c'}_i^{n+1} \right]  \rbrace \right)=$$
$$ \lbrace \left( p_{i,i+1}^{*}, {p'}_{i,i+1}^{*} \right) \left( \left[ c_{i+1}^n, {c'}_{i+1}^{n+1} \right] \right) - \left[ c_{i}^n, {c'}_{i}^{n+1} \right] \rbrace$$

$$ \forall  {\lbrace \left[ c_{i}^n, {c'}_{i}^{n+1} \right] \rbrace}  \in \prod H_{n+1}(f_i).$$
Therefore, $Coker \left( p_{*}, {p'}_{*} \right)={\varprojlim}^{1} H_{n+1}(f_{i})$  is first derivative of inverse limit of inverse sequence $\{ H_{n+1} \left( f_i \right) , \left(p_{i,i+1}^{*}, {p'}_{i,i+1}^{*} \right) , \mathbb{N} \} $. 
On the other hand, we have
$$
\left( p_{*},{p'}_{*} \right) \lbrace \left[ c_i^{n-1},{c'}_i^n \right] \rbrace=\lbrace \left( p_{i,i+1}^{*}, {p'}_{i,i+1}^{*} \right) \left[ c_{i+1}^{n-1}, {c'}_{i+1}^n \right] - \left[ c_{i}^{n-1}, {c'}_{i}^n \right] \rbrace,$$
$$ \forall  {\lbrace \left[ c_{i}^{n-1}, {c'}_{i}^n \right] \rbrace}  \in \prod H_{n}(f_i).$$
Consequently, $Ker \left( p_{*}, {p'}_{*} \right) = \varprojlim H_{n}(f_{i})$ is the inverse limit of inverse sequence $ \left( H_{n} \left( f_i \right) , \\ 
\{ p_{i,i+1}^{*}, q_{i,i+1}^{*} \right), \mathbb{N} \} $. The same way we will show that $Coker \left( q_{*}, {q'}_{*} \right)={\varprojlim}^{1} H_{n+1}(g_{i})$ and $Ker \left( q_{*}, {q'}_{*} \right) = \varprojlim H_{n}(g_{i})$ . Therefore, for each $n\in \mathbb{N}$ we have
$$
\begin{matrix}
   0\to Li{{m}^{1}}{{H}_{n+1}}( {{f}_{i}} )\to {{\mathbf{H}}_{\mathbf{n}}}( f )\to Lim~{{H}_{n}}( {{f}_{i}} )~\to 0  \\
   ~~~~~~~~~~~\downarrow Li{{m}^{1}} \left( {{\mathbf{\varphi_{i+1} }}},\mathbf{{\varphi'}_{i+1} } \right)_{*}~\downarrow \left( {{\mathbf{\varphi }}},\mathbf{{\varphi}' } \right)_{*}~~~~~\downarrow Lim \left( {{\mathbf{\varphi }}}_{i},{\mathbf{{\varphi}'}_{i} } \right)_{*} \\
   0\to Li{{m}^{1}}{{H}_{n+1}}( {{g}_{i}} )\to {{\mathbf{H}}_{\mathbf{n}}}( g )\to Lim~{{H}_{n}}( {{g}_{i}} )~\to 0  \\
\end{matrix}
.$$

\end{proof}

{\bf Conjecture.}  The homology functor  $\mathbf{\bar{H}}: \mathbf{Mo{{r}_{CM}}}\to \mathbf{Ab}$ from the category of continuous maps of compact metric spaces to the category of direct sequences of Abelian groups is unique up to isomorphism if and only if the following is fulfilled:

1) For each continuous map $f\in Mo{{r}_{CM}}$ of compact metric spaces there exist homomorphisms $c:{{\mathbf{ \bar{H}}}_{\mathbf{n}}}( C(f) )\to {{\mathbf{\bar{H}}}_{\mathbf{n}}}(  f  )$,  $\sigma _{*}^{'}:{{\mathbf{\bar{H}}}_{\mathbf{n}}}( {{X}'} )\to {{\mathbf{\bar{H}}}_{\mathbf{n}}}( C( f ) )$ and $~\partial _{*}^{'}:{{\mathbf{\bar{H}}}_{\mathbf{n}}}( C( f ) )\to {{\mathbf{\bar{H}}}_{\mathbf{n}-1}}( X )$ such that the following diagram is commutative:
$$
\begin{matrix}
   \cdots \to {{\mathbf{\bar{H}}}_{\mathbf{n}}}( {{X}'} )\overset{{{\sigma' }_{*}}}{\mathop{\to }}\,{{\mathbf{\bar{H}}}_{\mathbf{n}}}( C(f) )\overset{{{\partial' }_{*}}}{\mathop{\to }}\,{{\mathbf{\bar{H}}}_{\mathbf{n}-1}}( X )\overset{E}{\mathop{\to }}\,{{\mathbf{H}}_{\mathbf{n}-1}}( {{X}'} )\to \cdots \\
   ~~~~~~~\downarrow {id} ~~~~~~~~~~\downarrow c_{*}~~~~~~~~~~\downarrow {id}~~~~~~~~~\downarrow {id} ~~~~~~~ \\
   \cdots \to {{\mathbf{\bar{H}}}_{\mathbf{n}}}( {{X}'} )\overset{{{\sigma }_{*}}}{\mathop{\to }}\,{{\mathbf{\bar{H}}}_{\mathbf{n}}}( f )\overset{{{\partial }_{*}}}{\mathop{~\to~ }}\,{{\mathbf{\bar{H}}}_{\mathbf{n}-1}}( X )\overset{E}{\mathop{\to }}\,{{\mathbf{\bar{H}}}_{\mathbf{n}-1}}( {{X}'} )\to \cdots ,\\
\end{matrix}
$$
where $C(f)$ is cone of the continuous map $f:X \to Y$.

2) For each continuous map $f\in \mathbf{Mo{{r}_{CM}}}$ of compact metric spaces the corresponding homological sequence
	$$\cdots \to {{\mathbf{\bar{H}}}_{\mathbf{n}}}( {{X}'} )\overset{{{\sigma }_{*}}}{\mathop{\to }}\,{{\mathbf{\bar{H}}}_{\mathbf{n}}}( f )\overset{{{\partial }_{*}}}{\mathop{\to }}\,{{\mathbf{\bar{H}}}_{\mathbf{n}-1}}( X )\overset{E}{\mathop{\to }}\,{{\mathbf{\bar{H}}}_{\mathbf{n}-1}}( {{X}'} )\to \cdots $$
is exact.

3) If morphisms  $( {{\varphi }_{1}},\varphi _{1}^{'} ),( {{\varphi }_{2}},\varphi _{2}^{'} ):f\to g$ induce the same strong shape morphisms, then
$${{( {{\varphi }_{1}},\varphi _{1}^{'} )}_{*}}={{( {{\varphi }_{2}},\varphi _{2}^{'} )}_{*}}:\mathbf{\bar{H}}( f )\to \mathbf{\bar{H}}( g ).$$                 

4) If $( {\varphi }_,{\varphi'} ):f\to g$  is such a morphism that $f$ and $g$ are inverse limits of inverse ANR-sequences  $\mathbf{f}= \{ {{f}_{i}},({{p}_{i,i+1}},p_{i,i+1}^{'})~,\mathbb{N} \}$ and $\mathbf{g}= \{ {{g}_{i}},({{q}_{i,i+1}},q_{i,i+1}^{'})~,\mathbb{N} \}$ and the pair $( {\varphi }_,{\varphi'} ):f\to g$ is the pair of inverse limits on inverse ANR-sequence ${\mathbb {\varphi}}= \{ {{\varphi}_{i}},({{p}_{i,i+1}},q_{i,i+1}^{'})~,\mathbb{N} \} : { \mathbf{X} \to \mathbf{Y}}$ and ${\mathbf{\varphi'}}= \{ {{\varphi}_{i}},({{p}_{i,i+1}},q_{i,i+1}^{'})~,\mathbb{N} \}: { \mathbf{X'} \to \mathbf{Y'}}$ then, the following sequence 
$$
\begin{matrix}
   0\to Li{{m}^{1}}{{H}_{n+1}}( {{f}_{i}} )\to {{\mathbf{\bar{H}}}_{\mathbf{n}}}( f )\to Lim~{{H}_{n}}( {{f}_{i}} )~\to 0  \\
   ~~~~~~~\downarrow Li{{m}^{1}} \left( {{\mathbf{\varphi_{i+1} }}},\mathbf{{\varphi'}_{i+1} } \right)_{*}~\downarrow \left( {{\mathbf{\varphi }}},\mathbf{{\varphi}' } \right)_{*}~\downarrow Lim \left( {{\mathbf{\varphi }}}_{i},{\mathbf{{\varphi}'}_{i} } \right)_{*} \\
   0\to Li{{m}^{1}}{{H}_{n+1}}( {{g}_{i}} )\to {{\mathbf{\bar{H}}}_{\mathbf{n}}}( g )\to Lim~{{H}_{n}}( {{g}_{i}} )~\to 0  \\
\end{matrix}
$$
is exact.

\end{document}